\documentclass{article}

\usepackage{amsfonts}
\usepackage{amssymb}
\usepackage{amsthm}
\usepackage{amsmath}
\usepackage{mathrsfs}
\usepackage{stmaryrd}
\usepackage[latin1]{inputenc}
\usepackage[all]{xy}
\usepackage[left=1.20in,right=1.20in]{geometry}
\usepackage{array}

\usepackage{amsmath}

\def\pts{\,:\,}
\def\calt{\mathcal{T}}

\theoremstyle{plain}
\newtheorem{thm}{Theorem}[section]
\newtheorem{prop}[thm]{Proposition}
\newtheorem{cor}[thm]{Corollary}
\newtheorem{lem}[thm]{Lemma}

\theoremstyle{definition}

\theoremstyle{remark}

\newtheorem{rem}[thm]{Remark}

\usepackage{xy}
\xyoption{all}

\def\calf{\mathcal{F}}
\def\ent{\mathbb{Z}}

\def\map{\rightarrow}
\def\epi[#1]{
\xymatrix@C=#1pt{
\ar@{->>}[r] &
}
}

\title{Finite metacyclic groups as active sums of cyclic subgroups\\ 
Les groupes finis m\'etacycliques comme sommes actives de sous-groupes cycliques } 
\author{A. D\'iaz-Barriga$^*$, F. Gonz\'alez-Acu\~na$^*$, F. Marmolejo\footnote{Instituto de Matem\'aticas, UNAM, Mexico} $\ $and
N. Romero\footnote{MATHGEOM, EPFL, Switzerland}}
\date{}

\begin{document}


\maketitle
\begin{abstract}
The notion of active sum provides an analogue for groups of that of direct sum for abelian groups. One natural question then is which groups are the active sum of a family of cyclic subgroups. Many groups have been found to give a positive answer to this question, while the case of finite metacyclic groups remained unknown. In this note we show that every finite metacyclic group can be recovered as the active sum of a discrete family of cyclic subgroups.
\end{abstract}


\renewcommand{\abstractname}{R\'esum\'e}

\begin{abstract}
La notion de somme active fournit un analogue pour les groupes de celle de somme directe pour les groupes ab\'eliens. Une question naturelle est alors de d\'eterminer quels groupes sont somme active d'une famille de sous-groupes cycliques. De nombreux groupes possèdent cette propri\'et\' e, mais la question demeurait ouverte pour les groupes finis m\'etacycliques. Dans cette note, nous montrons que tout groupe fini m\'etacyclique s'obtient comme somme active d'une famille discrète de sous-groupes cycliques.
\end{abstract}

\section{Introduction}

The active sum of groups has its origin in a paper of Tom\'as \cite{tomasana}, where one of the main
motivations was to find an analogue of the direct sum of groups, but this time taking into account the mutual actions of the groups 
in question. For an arbitrary group $G$, the active sum of a generating family of subgroups, closed under conjugation and with partial order compatible with inclusion, is a group $S$ which has $G$ as an homomorphic image and that coincides with the direct sum of the family in case of $G$ being abelian. 
Since every finite abelian group is the direct sum of cyclic subgroups, it is natural to ask whether every
finite group is the active sum of cyclic subgroups. In \cite{sumasact}
it is shown that free groups, semidirect products of two finite cyclic groups, most of
the groups of the form $SL_n(q)$, and Coxeter groups are all active sums of discrete families of cyclic subgroups.  
On the other hand, in \cite{sumasact2} we gave examples of finite groups  
that are not active sums of cyclic subgroups: the alternating groups $A_n$ for $n\geq 4$, many of the groups
of the form $PSL_n(q)$, and others. However, at the time we were unable to determine whether every finite metacyclic group 
is the active sum of a discrete family of cyclic subgroups, a question we settle in this paper.

Dealing with the definition of active sum in concrete examples is not an easy task. Fortunately, one of its main properties is that for any group $G$ and any family $\mathcal{F}$ satisfying the conditions mentioned before, the active sum $S$ of $\mathcal{F}$ satisfies $S/Z(S)\cong G/Z(G)$ ($Z(S)$ and $Z(G)$ denote the center of $S$ and $G$
respectively).
This allows us to obtain conditions of homological nature, namely the surjectivity of Ganea's homomorphism, that helps us decide whether $S$ is isomorphic to $G$. As we will see in Section 3, in case of  $G$ being finite metacyclic, this reduces the problem to verify that the family proposed is  \textit{regular} and \textit{independent}. The notions of  regularity and independence will be recalled in the next section, while Ganea's homomorphism will be described in Section 3.

A group $G$ is {\em metacyclic} if it has a normal cyclic subgroup $K$ such
that $G/K$ is also cyclic. 
For a finite metacyclic group $G$, Hyo-Seob Sim \cite{sim} introduces the \textit{Standard Hall decomposition} of a given metacyclic factorization, splitting $G$ as a semidirect product of two groups $N$ and $H$ of relatively prime order. $N$ turns out to be a semidirect product of two cyclic groups and $H$ is nilpotent. By doing this, all the difficulties that arise when $G$ is not a semidirect product of two cyclic groups are gathered in the nilpotent group $H$. This is precisely what we will use to show that every finite metacyclic group is an active sum of cyclic subgroups. As we mentioned above, we know this to be true for the semidirect product of two cyclic groups; we will prove it for finite metacyclic $p$-groups, and we will finally use the Standard Hall decomposition to prove it for any finite metacyclic group.

The active sum shares some nice properties  with other similar constructions, such as the cellular cover of a group (see for example \cite{chac} or\cite{farj}). In fact, it seems quite plausible that these two constructions are closely related. For example, if the family contains only cyclic groups of order $n$, one can show that the active sum is $\ent/n\ent$-cellular, in the sense of Definition 2.2 of \cite{chac}. As a consequence, we have that Coxeter groups are $\ent/2\ent$-cellular, since any Coxeter group is the active sum of a family of groups of order 2. For more information about the active sum, and its relation to other areas, see the introduction to
\cite{sumasact}. 

\section{Preliminaries}

\subsection{Active sum}

We take as our definition of active sum the one given in \cite{sumasact}. Since in this paper we will
only consider discrete families of distinct subgroups, for the convenience of the reader we briefly describe the results
we need in this particular setting. Thus we take a (finite) group $G$, and a family $\calf$ of distinct subgroups of $G$
that is generating ($\langle\bigcup_{F\in\calf}F\rangle=G$) and closed under conjugation ($\forall F\in
\calf, g\in G, F^g=g^{-1}Fg\in \calf$). The {\em active sum} $S$ of $\calf$ is the free product of the elements of $\calf$
divided by the normal subgroup generated by the elements of the form $h^{-1}\cdot g\cdot h\cdot (g^h)^{-1}$, with $h\in F_1$, 
$g\in F_2$,  $F_1, F_2\in\calf$ (and thus, $g^h\in F_2^h=h^{-1}F_2h\in\calf$). We obtain a canonical homomorphism
$\varphi\pts S\to G$, surjective since $\calf$ is generating. 
By Lemma 1.5 of \cite{sumasact},  $\varphi^{-1}(Z(G))=Z(S)$, so we
obtain a central extension
\[
\xy
(0,0)*+{H_2(S)}="1",
(20,0)*+{H_2(G)}="2",
(40,0)*+{\textup{ker\,}\varphi}="3",
(60,0)*+{H_1(S)}="4",
(80,0)*+{H_1(G)}="5",
(100,0)*+{1}="6",
\POS "1" \ar^{\varphi_*} "2",
\POS "2" \ar "3",
\POS "3" \ar "4",
\POS "4" \ar^{\varphi_*} "5",
\POS "5" \ar "6",
\endxy .
\]
From Theorem 1.12 in \cite{sumasact} we have that $\varphi_*\pts H_1(S)\to H_1(G)$ is injective (and thus an iso) iff $\calf$ is  regular and independent, notions we will recall below. On the other hand, we will see in Section 3 that when $G$ is a metacyclic group, the arrow $\varphi_*\pts H_2(S)\to H_2(G)$ is always surjective. As a consequence, in the particular case
of $G$ being metacyclic, $\varphi\pts S\to G$ is an isomorphism iff $\calf$ is regular and independent. 

The notions of regularity and independence are explained in more detail in Section 1.2. of \cite{sumasact}. For regularity we recall Lemma 1.9 of this reference:

\begin{lem}
Assume that the family $\calf$ is discrete. The family $\calf$ is regular iff
$[F,N_G(F)]=F\cap G'$ for every $F\in\calf$, where $N_G(F)$ is the normalizer of
$F$ in $G$.
\end{lem}

As for independence, we consider first a complete set of representatives $\calt$
of conjugacy classes of elements of $\calf$. We call $\calt$ a transversal of $\calf$. 
According to Definition 1.11 of \cite{sumasact}, $\calf$ is independent iff the canonical
homomorphism
\[
\bigoplus_{F\in \calt}F/(F\cap G')\to G/G'
\]
is an isomorphism.

\subsection{Metacyclic groups}

A group $G$ is called metacyclic if it has a normal cyclic subgroup $K$ with cyclic quotient. 
By taking generators $a$ of $K$ and $bK$ of $G/K$, $G$ has a presentation of the form.
\begin{equation}\label{presentacion}
G=\langle a,b|a^m=1, b^s=a^t, b^{-1}ab=a^r\rangle
\end{equation}
where the integers $m$, $s$, $t$ and $r$ satisfy $r^s\equiv 1\ (\textsl{mod}\ m)$ and $m|t(r-1)$.
Proposition 1 of Chapter 7 in Johnson \cite{johnson} shows that a group is metacyclic with cyclic group of order $m$ and quotient group of order  $s$ if and only if it has a presentation of this form. If we let $L=\langle b\rangle$, then $G$ can be written as $G=LK$. This decomposition is called a \textit{metacyclic factorization} of $G$. If $L\cap K=1$, then the metacyclic factorization is called \textit{split}. 

In the rest of the section we will fix $m,\, s,\, t$ and $r$ as the integers of the above presentation, and $G=LK$ 
as above will be a fixed metacyclic factorization of $G$.

\begin{lem}\label{above}
Let $s'$ be the order of $r$ in $\ent_m^*$, and $k=\frac{m}{(r-1,m)}$.
The center of $G$ is the group $Z(G)=\langle a^k,b^{s'}\rangle$ and the Schur multiplier of $G/Z(G)$ is cyclic of order $q/k$ where $q=(r-1,\, k)(\frac{r^{s'}-1}{r-1},\, k)$.
\end{lem}
\begin{proof} Since $b^{-s'}ab^{s'}=a^{r^{s'}}=a$, it is clear that $b^{s'}\in Z(G)$.
Furthermore, if $a^\ell\in Z(G)$, then $a^\ell=b^{-1}a^\ell b=
a^{\ell r}$. Therefore $m|\ell(r-1)$. The smallest possible positive value
for $\ell$ is $k=\frac{m}{(r-1,\, m)}$. Thus 
$\langle a^k,\, b^{s'}\rangle\subseteq Z(G)$.

On the other hand, if $a^ub^v\in Z(G)$, then 
$b=(a^ub^v)b(a^ub^v)^{-1}=a^uba^{-u}$, therefore $k|u$. Similarly
$a^ub^v=a(a^ub^v)a^{-1}=a^u(ab^va^{-1})$, thus $b^v=ab^va^{-1}$. This means
that $s'|v$.

Since any element of the form $a^p=b^q$ is in $Z(G)$, it is easy to see that $G/Z(G)$ is a semidirect product, therefore it has the following presentation 
\[
G/Z(G)\simeq \langle c,d|c^{k}=1,d^{s'}=1, d^{-1}cd=c^r\rangle .
\]
By Theorem 2.11.3 in $\cite{karpi}$, the Schur multiplier of $G/Z(G)$ is
cyclic of order $\frac{(r-1,k)(\frac{r^{s'}-1}{r-1},k)}{k}$.
\end{proof}

\begin{rem}\label{above2}
It is easy to see that $G'=\langle a^{r-1}\rangle$, so
we also have  $G'\cap Z(G)=\langle a^{r-1}\rangle\cap 
\langle a^k,b^{s'}\rangle=\langle a^{\frac{m}{(r-1,k)}}\rangle$. 
Therefore $|G'\cap Z(G)|=(r-1,k)$.
\end{rem}

\section{Ganea's map, regularity and independence}

Recall from Section 2.5 of \cite{sumasact}
the  exact sequence given by Ganea \cite{ganea}
for $G$ and $Z(G)$:
\[
H_1G\otimes Z(G) \map 
H_2G \map 
H_2(G/Z(G)) \map 
Z(G) \map 
H_1G \map 
H_1(G/Z(G)) \map  0.
\]
Observe that $\textsl{ker}(Z(G)\map H_1G)=Z(G)\cap G'$. Therefore we
obtain the following exact sequence
\[
H_1G\otimes Z(G) \map 
H_2G \map 
H_2(G/Z(G)) \twoheadrightarrow 
Z(G)\cap G'.
\]
Thus, the homomorphism $H_1G\otimes Z(G) \map H_2G$ is epi if and only if
the epimorphism $H_2(G/Z(G)) \map Z(G)\cap G'$ is also mono. This happens if and only if
$|H_2(G/Z(G))|\leq |G'\cap Z(G)|$. In case $G$ is a finite metacyclic group, considering the parameters given in the presentation (\ref{presentacion}), it follows from Lemma \ref{above} and
Remark \ref{above2}, that 
this happens if and only if
\[
\left(\frac{r^{s'}-1}{r-1},k\right)\leq k,
\]
but this is clear. 
We have shown

\begin{prop}
If $G$ is a finite metacyclic group, then Ganea's homomorphism
$H_1G\otimes Z(G)\map H_2G$ is an epimorphism.
\end{prop}


Theorem 2.16 in \cite{sumasact} says that if Ganea's homomorphism is surjective, then to prove that the active sum of a family $\mathcal{F}$ is isomorphic to $G$, it is enough to show that $\mathcal{F}$ is regular and independent. Thus we obtain

\begin{cor}\label{regindepmeta}
If $G$ is a finite metacyclic group, then every generating regular and independent
family of subgroups of $G$ has active sum isomorphic to $G$.
\end{cor}

Our task now is to show that every finite metacyclic group has a regular
and independent family composed solely of cyclic subgroups.

For the rest of the section we will assume that $G$ is a metacyclic group with a presentation (\ref{presentacion}).

\medskip

\begin{rem} We may assume that $G/G'$ is not cyclic since it is easy to prove that
$G/G'$ cyclic implies that $G$ is a semidirect product of two cyclic
groups.
\end{rem}

\begin{lem}
\label{condindepcoro}
The family consisting of the groups $\langle a\rangle$,
$\langle b\rangle$ and their conjugates is generating and 
independent for $G$ if and only if
$(m,\, r-1)|t.$
\end{lem}
\begin{proof}
The family is clearly generating. 
The quotient group $G/G'$ is an abelian, non-cyclic metacyclic group, which means it is the direct sum of two cyclic groups.
To ask for the family to be independent amounts to ask for 
\begin{displaymath}
\frac{G}{G'}= \frac{\langle a\rangle}{G'}\oplus \frac{\langle b\rangle G'}{G'}.
\end{displaymath}
Now, the order of $G/G'$ is $(m,\, r-1)s$, while the order of $\langle a\rangle/G'$ is $(m,\, r-1)$, so the above equality holds if and only if $\langle b\rangle/(\langle b\rangle \cap G')$ has order $s$. Since $s$ is the smallest integer such that $b^s\in\langle a\rangle$, this happens iff $\langle b\rangle \cap G'=\langle b^s\rangle$ which happens iff $(m,\, r-1)|t$.
\end{proof}

\begin{thm}
If $(m,r-1)|t$, then the metacyclic group $G$ is the active sum of the family
consisting of $\langle a\rangle$, $\langle b\rangle$ and its conjugates.
\end{thm}
\begin{proof}
By the two previous results we only need to show that the family is regular

Observe that $\langle a\rangle$ is always regular since
$G'\subseteq \langle a\rangle$.
We would like to show that $\langle b\rangle$ is regular as well. That is,
we need to see that
$G'\cap \langle b\rangle\subseteq[\langle b\rangle,N_G(\langle b\rangle)]$.
We have $G'\cap \langle b\rangle =\langle b^s\rangle$.
Therefore, we need to show that $b^s\in 
[\langle b\rangle,N_G(\langle b\rangle)]$.

Using the relation $a^b=a^r$ we deduce $b^{a^{-1}}=ba^{r-1}$.
Suppose that  $t=q(m, \, r-1)$ and that $(m,\, r-1)=\alpha m+\beta (r-1)$. If we let $z=-q\beta$ then $ba^t=ba^{-(r-1)z}$, which by the previous equality gives $b^{a^{z}}=ba^t=b^{s+1}$.
We thus have $a^z\in N_G(\langle b\rangle)$ and $b^s=[b,a^{z}]\in 
[\langle b\rangle,N_G(\langle b\rangle)]$, and conclude that
$\langle b\rangle$ is regular. 
\end{proof}

\section{Metacyclic groups as active sums of cyclic subgroups}

We apply the previous theorem to the case of metacyclic
$p$-groups. To do this we use the presentation given on Theorem 3.5  of Sim \cite{sim} for the case $p$ odd, and the presentations of Theorem 4.5 in Hempel \cite{hempel} for the case $p=2$. A direct inspection of all these presentations,  shows that they are either abelian or of the form 
(\ref{presentacion}), that is, they
satisfy the numerical conditions given for (\ref{presentacion}). Furthermore,
it is immediate that, discarding the abelian ones, they satisfy 
the condition of the previous theorem, so we have

%

\begin{thm}\label{allpmeta}
If $G$ is a finite metacyclic $p$-group, then it is the active sum of  the family formed by
$\langle a\rangle$ and the conjugates of $\langle b\rangle$ in $G$
(as given in the above-mentioned presentations).
\endproof
\end{thm}

We now consider the general case. That is, $G$  will be a finite metacyclic group with a presentation of the form (\ref{presentacion}) and we fix $K=\langle a\rangle$ and $L=\langle b\rangle$ as a metacyclic factorization $G=LK$ for $G$. From section 5 of \cite{sim}, we have that $G$ has a decomposition $G=N\rtimes H$ where $N$ and $H$ are Hall subgroups with $H$ nilpotent. This decomposition also satisfies that $N=(L\cap N)(K\cap N)$ is a split metacyclic factorization and has following properties:

\begin{lem}
\label{props}
Let $U=L\cap N$ and $V=K\cap N$, we have
\begin{itemize}
\item[\textup{i)}] $U\leqslant \mathbb{C}_N(H):=\{n\in N|\forall h\in H(hn=nh)\}$.
\item[\textup{ii)}] $G'\cap N=V$ and $G'\cap H=H'$.
\end{itemize}
\end{lem}
\begin{proof}
i) and the first part of ii) are i) and ii) of Lemma 5.6 in \cite{sim}. The second part of ii) is not difficult to prove and so it is left to the reader. 
\end{proof}

Since $N=V\rtimes U$ is a splitting metacyclic factorization, if we let $V=\langle v\rangle$ and $U=\langle u\rangle$, then $N$ has a presentation
\begin{displaymath}
\langle u,\, v\mid v^{\zeta}=u^{\epsilon}=1,\, v^u=v^{\eta} \rangle,
\end{displaymath}
with ${\eta}^{\epsilon}\equiv 1$ mod $\zeta$. By \cite{sumasact} we know that $N$ is the active sum of the family $\mathcal{F}_N=\langle F_i\rangle_{i=0}^{\zeta}$ with $F_0=\langle v \rangle$ and $F_i=\langle uv^{i(\eta-1)}\rangle$ if $i>0$.

On the other hand, $H$ is the direct product of its $p$-Sylow subgroups $H_p$ with $p\in \pi=\pi(H)$. According to  Proposition \ref{allpmeta}, $H_p$ has a metacyclic factorization $B_pA_p$ (with $A_p\trianglelefteq H_p$) such that the family $\mathcal{F}_{H_p}$, consisting of $A_p$, $B_p$ and the conjugates of the latter in $H_p$, is regular and independent. Now consider the family $\mathcal{F}_H=\bigcup_{p\in \pi}\mathcal{F}_{H_p}$.

\begin{lem}
$H$ is the active sum of $\mathcal{F}_H$.
\end{lem} 
\begin{proof} Simply observe that the action of $H_p$ over $H_{p'}$ with $p\neq p'$ is trivial.
\end{proof}

For $G$ we consider then the family $\mathcal{F}$ formed by closing $\mathcal{F}_N\cup \mathcal{F}_H$ under conjugation by $G$.

It is easy to see that the family $\mathcal{F}$ 
consists of $\mathcal{F}_N$ 
together with the conjugates of $A_p$ by $V$ and the conjugates of $B_p$ by $VH_p$. We also observe that $\mathcal{T}=\{U,\, V\}\cup \bigcup_{p\in \pi}\{A_p,\, B_p\}$ is a transversal for $\mathcal{F}$. It suffices to prove that in the case $H_p$ non-cyclic, $B_p$ is not equal to $A_p^{v}$ for some $v$ in $V$. Suppose that $B_p=A_p^{v}$. If we let $\langle b\rangle =B_p$, this would imply $b=vav^{-1}$ for some $a\in A_p$. But $V$ is normal in $G$ so we have $b=av'$ for some $v'\in V$. Since $N\cap H=1$, this implies $b=a$, which  contradicts the assumption of $H_p$ non-cyclic.

\begin{thm}
\label{allmeta}
$G$ is the active sum of the family $\mathcal{F}$.
\end{thm}
\begin{proof}
It suffices to show that $\mathcal{F}$ is regular and independent.

We first prove regularity. We must show that for every $F$ in $\mathcal{T}$ we have $[F,\, N_G(F)]=F\cap G'$.

\begin{itemize}
\item[a)] $F=V$.

Since $N\cap G'=V$, we have $V\cap G'=V$. On the other hand, $N_G(V)=G$, so $[V,\, N_G(V)]=[V,\, G]$. Now, for any $C\leqslant K$ we have $[C,\, G]=[C,\, L]$. Also, if $C=\langle c\rangle$, presentation (\ref{presentacion}) gives $[C,\, L]=\langle c^{r-1}\rangle$. That is, $[V,\, G]=\langle v^{r-1}\rangle$ if $v$ is a generator of $V$. Let us see that $(r-1,\, |V|)=1$. Suppose there exists a prime $p$ that divides $r-1$ and $|V|$. Then $p$ divides $|K|$ and we have $p\mid (r-1,\, |K|)$. On the other hand $|G'|$ is equal to $|K|/(r-1,\, |K|)$, so the exponent of $p$ in $|G'|$ is strictly smaller than its exponent on $|K|$. This is a contradiction since $|V|$ divides $|G'|$ and the exponent of $p$ in $|V|$ is the same as the one in $|K|$, because $V$ is a Hall subgroup of $K$.

\item[b)] $F=U$

Since $N\cap G'=V$, we have $U\cap G'=1$.

\item[c)] $F=A_p$ or $F=B_p$ for some $p$ in $\pi$.

By the regularity of $\mathcal{F}_H$ we have
\begin{displaymath}
F\cap H'=[F,\, N_H(F)]\subseteq [F,\, N_G(F)]\subseteq F\cap G',
\end{displaymath}
but by Lemma \ref{props} $F\cap G'= F\cap H'$.
\end{itemize}

For independence we consider again the transversal $\mathcal{T}$. 

We prove first that $G/G'$ is isomorphic to $U\oplus (H/H')$. Since $N$ is the semidirect product of $U$ and $V$ and $G$ is the semidirect product of $N$ and $H$, each element $g\in G$ can be written in a unique way as $uvh$ for some $u\in U$, $v\in V$ and $h\in H$. Thus we can define a function $f:G\rightarrow U\oplus H/H'$ that sends $g$ to $u+hH'$. Now take $g_1g_2\in G$ and write $g_i=u_iv_ih_i$ for $i=1,\,2$. Since $U\leqslant \mathbb{C}_N(H)$, we have $g_1g_2=u_1v_1u_2h_1v_2h_2$, but also $V$ is a normal subgroup of $G$, so this is equal to $u_1u_2v'h_1h_2$ for some $v'\in V$ and $f$ is a morphism of groups. Finally, if $f(uvh)=1$, then $u=1$ and $vh\in G'$, which gives  the isomorphism. Now, Lemma \ref{props} gives us $UG'/G'\cong U$ and $VG'/G'\cong 1$. 
On the other hand, since $\mathcal{F}_H$ is independent and $\bigcup_{p\in \pi}\{A_p,\, B_p\}$ is a transversal for it, we have 
\begin{displaymath}
H/H'\cong \bigoplus_{p\in \pi}A_pH'/H'\oplus B_pH'/H', 
\end{displaymath}
but $H\cap G'=H'$ implies that if $F_p=A_p$ or $B_p$, then $F_pH'/H'\cong F_pG'/G'$.
\end{proof}

\bibliography{sumasIIIc.bib}{}
\bibliographystyle{plain}

\begin{flushleft}
E-mail: \texttt{diazb$\, $@$\, $matem.unam.mx,  fico$\, $@$\, $matem.unam.mx}, 

$\quad \quad \quad \ \, $\texttt{quico$\, $@$\, $matem.unam.mx, nadiaro$\, $@$\, $ciencias.unam.mx} 
\end{flushleft}

\end{document}